\documentclass[10pt]{amsart}
\usepackage[]{amsmath, amsthm, amsfonts,amssymb}
\usepackage{graphicx}
\usepackage[all]{xy}
\usepackage[colorlinks=true]{hyperref}

\newcommand {\N} {{\mathbb N}}
\newcommand {\C} {{\mathbb C}}
\newcommand {\R} {{\mathbb R}}

\newcommand {\Z} {{\mathbb Z}}

\newcommand {\OO} {{\mathcal O}}

\newcommand {\dt} {\bullet}
\newcommand {\cH} {\mathcal{H}}

\newcommand {\X} {\mathcal{X}}
\newcommand {\Y} {\mathcal{Y}}

\DeclareMathOperator{\Spec}{Spec}

\DeclareMathOperator{\dlog }{dlog}

\newtheorem{thm}[subsection]{Theorem}

\newtheorem{lemma}[subsection]{Lemma}

\newtheorem{remark}[subsection]{Remark}
\newtheorem{ex}[subsection]{Example}

\begin{document}
\title{Decomposition of direct images of  logarithmic differentials}

\author{ Donu Arapura}
\address{Department of Mathematics\\
Purdue University\\
West Lafayette, IN 47907\\
U.S.A.}
\thanks{Author partially supported by the NSF}
\email{dvb@math.purdue.edu}

\maketitle

In \cite{kollar}, Koll\'ar proved a remarkable theorem that given a
projective map $f$ from a smooth complex 
algebraic variety $X$ to an arbitrary variety $Y$, the derived direct image of the canonical sheaf $\R f_*\omega_X$ decomposes
as a sum $\bigoplus R^if_*\omega_X[-i]$. So in particular, the Leray
spectral sequence for $\omega_X$ degenerates.  Saito
\cite{saito}  gave a  second proof using his
theory of Hodge modules, which gives a good conceptual explanation 
for the theorem. Although the prerequisites for understanding it are
rather heavy. A third analytic
proof, of the degeneration at least, was given by Takegoshi \cite{takegoshi}.
The purpose of this note is to give another proof of the full theorem, which is fairly
short.  The basic idea is as follows. The weak semistable reduction theorem of Abramovich and Karu \cite{ak}, can be
interpreted as saying that any map can be converted, in the
appropriate sense, to a map which is particularly nice from the point
of view of logarithmic geometry \cite{kato}. For such maps, we prove
that derived direct images of the sheaves of logarithmic
differentials  decompose as above. It is important that we work with
 differentials of all degrees simultaneously, because the sum of  these direct images carries a Lefschetz
operator. The proof of the decomposition theorem ultimately comes down to
checking that this sum satisfies  the hard Lefschetz theorem and then
applying Deligne's
theorem \cite{deligne}.  The final step is to show that this  theorem
implies Koll\'ar's. 

The core arguments, which  take only three  pages, are given in
the second section. The  first section contains some background
material on logarithmic  geometry, which is included in the interest of making
this more accessible. We work over $\C$. The words morphism and map are used
interchangeably, according to whim. 

 I owe a special thanks to the principal referee for pointing out a number of
issues with the original exposition.
My thanks also to the other
referees  and T. Fujisawa and for various other comments.

\section{Some  log geometry}

In this expository section, we summarize some facts
about log schemes needed later. The details can be found in \cite{kato,
  ogus} together with some more specific references given below.  For our purposes, the basic
example is  given by a pair consisting of a smooth
 scheme $X$ and  a reduced divisor $E\subset X$ with normal crossings.
We refer to this as a {\em log pair}. We can segue to the more
flexible  notion of log structure by noting that  a log pair $(X,E)$,
gives rise to the
multiplicative  submonoid  $M\subset \OO_X$ of    functions invertible
outside of $E$. 
A {\em pre-log structure} on a scheme $X$ is a sheaf of commutative monoids
$M$ on the \'etale topology $X_{et}$ together with a homomorphism $\alpha:M\to \OO_X$, where the
latter is treated as a monoid with respect to multiplication. It is a
{\em log structure} if $\alpha^{-1}\OO_X^*\cong
\OO_X^*$.  For example, the  monoid associated to a log pair is a log
structure. In general, a  pre-log structure can be completed to a log structure in
a canonical way. A {\em log scheme} is a scheme equipped with a  log
structure. We will identify log pairs with the associated log scheme. In
order to avoid confusion below, we reserve the symbols $\X, \Y,
\ldots$ for log schemes, and use the corresponding symbols $X,Y,\ldots$ for the
underlying schemes. There are a number of other examples in addition to
the one given above. For example, any scheme $X$ can be turned into a log scheme with the {\em
    trivial} log structure $M=\OO_X^*$.

There is an important  connection between log  geometry and  toric
geometry. Indeed, given a finitely generated saturated submonoid $M\subseteq \Z^n$, the
affine toric variety $\Spec \C[M]$ is a log scheme
with respect to the pre-log structure induced by $M\to \C[M]$. 
More generally, recall that a  toroidal variety \cite{kkms} is given by a
variety $X$ and an open subset $U$, such that the pair $(X,U)$ is \'etale
locally isomorphic to a toric variety with its embedded torus. The toroidal
variety $X$ carries a log structure given locally by the one
above. The subset $U$ can be understood as the locus where this log
structure is trivial. In the log setting, it is convenient
to relax the condition on $M$ to allow it to be a
finitely generated monoid which is embeddable into an abelian group as a
saturated monoid.  If $M$ satisfies these conditions, it is called
{\em fine and saturated} or simply {\em fs}. Note that there is a canonical
choice for the ambient group, namely the group $M^{gp}$  given as a
group of fractions. In general, we
want to restrict our attention to log schemes (called fs log schemes)
for which the {\em characteristic monoids} $\overline{ M}_x=
M_x/\OO^*_x$ are fs for all geometric points $x$.

 A morphism of log schemes consists of a
morphism of schemes and a compatible morphisms of monoids.  Of course,
any morphism of schemes can be regarded as a morphisms of log schemes
when equipped with the trivial log structures. Here  are some more interesting examples.

\begin{ex}
Given log pairs $\X=(X,E)$ and $\Y = (Y,D)$.
  A semistable morphism of log schemes $f:\X\to \Y$ is a
morphism $f$ of schemes which is given \'etale locally (or
analytically)  by
$$y_1=x_1x_2\ldots x_{r_1}$$
$$y_2=x_{r_1+1}\ldots x_{r_2}$$
$$\ldots$$
$$y_{d+1}=x_{r_d+1}$$
$$\ldots$$
such that $x_1\ldots x_{r_d}$ and  $y_1\ldots y_{d}$ are the local
equations for $E$ and $D$ respectively.
\end{ex}

\begin{ex}\label{ex:chart}
  A homomorphism of fs monoids $P\to Q$  induces a
  morphism of fs log schemes $\Spec \C[Q]\to \Spec \C[P]$. In
  particular, toric morphisms of toric varieties are morphisms of log schemes.
\end{ex}

Log structures give rise to logarithmic differentials in a rather
general way. Given a log pair $\X=(X,E)$, set
$$\Omega_{\X}^k = \Omega_X^k(\log E)$$
which is generated locally by $k$-fold wedge products of
$$\frac{dx_1}{x_1},\ldots \frac{dx_{r_d}}{x_{r_d}}, dx_{r_d+1},\ldots$$
More generally, given a log scheme $\X$  over  $\C$, we define the $\OO_X$-module
$\Omega_\X^1=\Omega_{\X/\C}^1$ as the universal sheaf which receives  a $\C$-linear derivation $d:\OO_X\to
\Omega_{\X}^1$ and homomorphism $\dlog:M\to \Omega_{\X}^1$ satisfying
$m\dlog m = dm$.  Of course, by the universal property of ordinary
K\"ahler differentials, we have a map $\Omega_X^1\to\Omega_\X^1$
taking $df$ to $df$, but this is generally not an isomorphism.
There is a relative version  of differentials $\Omega_{\X/\Y}^1$ for a
morphism $f:\X\to \Y$ of log schemes. This  fits into an exact sequence
$$f^*\Omega_\Y^1\to \Omega_\X^1\to \Omega_{\X/\Y}^1\to 0$$

There is a notion of smoothness in this setting, called {\em log
  smoothness}, which can be defined using a variant of the usual
infinitesimal lifting condition \cite[\S 3.3]{kato}.
However, it is weaker than the
name suggests. For instance,
while smoothness implies flatness, the corresponding statement for log
smoothness is false. Nevertheless, some expected properties do hold.
 For a log smooth map
$\Omega_{\X/\Y}^1$ and therefore its exterior powers $\Omega_{\X/\Y}^k$,
are locally free. Kato \cite[thm 3.5]{kato} gives a criterion which
allows us to verify  some basic examples:

\begin{ex}
Toric  (and more generally toroidal) varieties are log smooth over $\Spec \C$
with its trivial log structure. 
\end{ex}

\begin{ex}\label{ex:semi}
 A semistable map between log pairs is a log smooth morphism. 
\end{ex}

To rectify some of the defects of log smoothness just alluded to,  we need a few more
conditions. A morphism of fs  monoids $h:P\to Q$ is exact if satisfies $(h^{gp})^{-1}(Q)= P$;
it is integral if $\Z[P]\to \Z[Q]$ is flat; it is saturated if for any
morphism $P\to P'$ to an fs monoid the push out $Q\oplus_P P'$ is fs; and it is vertical if the  image of $P$
does not lie in any proper face of $Q$. For example, the diagonal
embedding $\N\subset \N^n$ satisfies all of these conditions.
A morphism of $f:(X,M)\to (Y,N)$ of fs log schemes,
 is respectively {\em exact}, {\em
  integral}, {\em saturated} or {\em vertical}, if the map of characteristic monoids
$(f^{-1}\overline{N})_x\to \overline{M}_x$ has the same property for each geometric
point $x$.
 Integral maps are exact and also flat as a map of
schemes. Saturated morphisms are integral with reduced fibres \cite[\S
3.6]{it}
Here are the key examples for us.

\begin{ex}\label{ex:semi2}
 A semistable map between log pairs is  vertical and saturated (and
 therefore integral and therefore exact). This is because the maps of characteristic monoids are sums
 of diagonal embeddings $\N^m\subset \N^{n_1+\ldots +n_m}$.
\end{ex}

\begin{ex}\label{ex:semi3}
Abramovich and Karu \cite[def 0.1]{ak} define a map from a  scheme to a
regular scheme to be weakly semistable  if  it is  toroidal and  equidimensional with reduced fibres.
Such a map is log smooth vertical and saturated when the
  schemes are endowed with the log structures  induced from the toroidal
  structures, cf  \cite[rmk 3.6.6]{it}. 
\end{ex}

We note that there is a parallel  theory of  log analytic
spaces given by an analytic space $(X,\OO_X)$ and a homomorphism of
monoids $\alpha:M\to \OO_X$ satisfying conditions similar to those
above. Most of the basic definitions and constrictions carry over as before.
In addition, there is a new
construction often called the {\em real blow up} \cite{kn}. 
Given  an fs log analytic space  $\X$, we can define a new topological
space $\X^{log}$ and a continuous map $\lambda:\X^{log}\to
X$. 
As a set $$\X^{log}=
\left\{ (x,h)\mid x\in X, h\in Hom(M^{gp}_{x}, S^1), \forall f\in \OO_{x},
h(f)= \frac{f(x)}{|f(x)|} \right\}
$$
and $\lambda$ is given by the evident projection.
For example, when $\X$ is given by the log pair $(\C^n,\{z_1\ldots z_n=0\})$, 
$\X^{log} = ([0,\infty)\times S^1)^n$ as a topological space, with $\lambda(r_1,
u_1,r_2,u_2,\ldots)\mapsto (r_1u_1, r_2u_2,\ldots)$. The construction
is functorial so that a morphism  $f:\X\to \Y$ gives a continuous map
$f^{log}:\X^{log}\to \Y^{log}$ compatible with the $\lambda$'s.
We note also that  $\X^{log}$ can be made into a ringed space with structure
sheaf $\OO_{\X^{log}}$ such that $\lambda$ becomes a morphism.
One can picture $(X,E)^{log}$ as adding an ideal boundary to $X-E$ which is homeomorphic to the boundary
of a tubular neighbourhood about $E$, and this picture is compatible
with  what is happening on $Y$. This leads to the remarkable fact, due
to Usui, that  when $f$ is proper and semistable $f^{log}$ is topologically a fibre bundle. A generalization of this due to Nakayama and Ogus  \cite{no} will be used below.

\section{Main theorems}

\begin{thm}\label{thm:decomp2}
  Let $f:\X\to \Y$ be a projective  exact vertical log smooth map  of fs log schemes and suppose
  that $\Y$ is   the log scheme associated to a
    log pair.
Then
$$\R f_* \Omega_{\X/\Y}^k \cong \bigoplus_i R^i
f_*\Omega_{\X/\Y}^k[-i]$$
for all $k$.
\end{thm}

Readers dismayed by this long string of adjectives should keep in mind that
a projective semistable map satisfies all of the above assumptions, and therefore
the conclusion. The semistable case, however, is not strong enough to
imply Koll\'ar's theorem. We really need the version just stated.
Before giving the proof, we record the following elementary fact,
which can be proved by induction on the length of the filtrations.

\begin{lemma}\label{lemma:lin}
  Suppose that $(V_i,F_i^\dt)$ are two filtered finite dimensional
  vector spaces such that $\dim Gr^p(V_1)=\dim Gr^p(V_2)$ for all $p$.
Then a linear isomorphism $L:V_1\to V_2$ which is compatible with the
filtrations will induce an isomorphism of associated graded spaces.
\end{lemma}

\begin{proof}[Proof of theorem]
Let $d$ denote the relative dimension of $f$. Let $\Y$ be defined by
the log pair $(Y,D)$. The verticality assumption implies that the log
structure of $\X$ is trivial on $U = X-f^{-1}D$. The restriction $f|_U:U\to
Y-D$ is smooth and projective in the usual sense. So it is
topologically a fibre bundle. In fact a theorem of Nakayama and Ogus \cite{no} shows
that $f|_U$ prolongs naturally to a fibre bundle $f^{log}:\X^{log}\to
\Y^{log}$ over $\Y^{log}$.

Using  the work of  Illusie, Kato and Nakayama
\cite[cor 7.2]{ikn}, we may  conclude that sheaves $R^i f_*\Omega_{\X/\Y}^k$ are
 locally free and that the relative  Hodge to de Rham spectral sequence
 degenerates.  Let
$\eta'\in H^1(X,\Omega_X^1)$ denote   $c_1$ of a
relatively ample  line bundle, and
let  $\eta\in H^1(\Omega_{\X/\Y}^1)\cong
Hom_{D(Y)}(\OO_Y,\R f_*\Omega^1_{\X/\Y}[1])$ denote the image of
$\eta'$. Let $V^k=\R f_*\Omega_{\X/\Y}^k$ and $V=\bigoplus_k V^k[-k]$.
  We have a Lefschetz operator $L:V^k\to V^{k+1}[1]$
given by cup product  with $\eta$.
By adding these, we get a map $L:V\to V[2]$. Our goal  is to establish
the hard Lefschetz property that $L^i$ induces an isomorphism
\begin{equation}
  \label{eq:HL}
  L^i:\mathcal{H}^{d-i}(V)\cong \mathcal{H}^{d+i}(V)
\end{equation}
for all $i$. Then the theorem will follow from Deligne's
theorem \cite{deligne}.  
Note that we have canonical isomorphisms
\begin{equation}
  \label{eq:cHV}
\cH^i(V)_y \cong \bigoplus_k   (R^{i-k}f_*\Omega^k_{\X/\Y} )_y\cong
\bigoplus_k (Gr_F^k\R^i f_*\Omega_{\X/\Y}^\dt)_y 
\end{equation}
where $F$ is the Hodge filtration. This isomorphism respects the action by $L$.
We note also that the ranks of $R^{d-i-k}f_*\Omega^k_{\X/\Y}$ and
$R^{d-k}f_*\Omega^{k+i}_{\X/\Y}$ coincide with the ranks of
$R^{d-i-k}f_*\Omega^k_{U/Y-D}$ and $R^{d-k}f_*\Omega^{k+i}_{U/Y-D}$
respectively, and therefore with each other. Therefore
by lemma \ref{lemma:lin} and Nakayama's lemma, it suffices to check
the hard Lefschetz property for $(\R^i f_*\Omega_{\X/\Y}^\dt)_y$ and
all $y\in Y$.
By \cite{fkato} , we have a canonical identification 
\begin{equation}
  \label{eq:RfOmega}
\lambda^*\R f_*\Omega_{\X/\Y}^\dt\cong  \R f^{log}_*\C\otimes \OO_{Y^{log}}
\end{equation}
Note that $\R f^{log}_*\C$ also has  an action by $L$, and \eqref{eq:RfOmega} respects these actions.
Since the stalk $(R^if^{log}_*\C)_y$ is just $H^i(X_y,\C)$, when $y\notin D$, and $R^i f^{log}_*\C$ is a local system,
we can conclude that we have a hard Lefschetz theorem everywhere,
i.e. $L^i: R^{d-i} f^{log}_*\C\cong  R^{d+i}f^{log}_*\C$. By the
previous remarks, this implies \eqref{eq:HL} and consequently the theorem.
\end{proof}

\begin{remark}
  The above argument can be pushed a bit to imply the same
  decomposition for a proper holomorphic semistable map
 of analytic log pairs  provided that there is a $2$-form on $X$ which
 restricts to a K\"ahler form on the components of all the fibres.
For this  we may appeal to a
  theorem of Fujisawa \cite[thm 6.10]{fujisawa} to conclude $R^i f_*\Omega_{\X/\Y}^k$ are
 locally free and that the relative  Hodge to de Rham spectral
 degenerates. The rest of the argument is the same as above.
\end{remark}

Before turning to Koll\'ar's theorem, we need the following:

\begin{lemma}\label{lemma:2}
Suppose that  $\Y=(Y,D)$ is a log pair, $X$ is a variety with  normal
Gorenstein singularities, and that $f:X\to Y$ is map such that the log
scheme $\X$ defined by $f^{-1}D$ is fs. We suppose furthermore that
 $f:\X\to \Y$ is log smooth and
saturated. Then
  $\Omega_{\X/\Y}^d\cong \omega_{X/Y}$, where $d=\dim X-\dim Y$ and
  $\omega_{X/Y}$ is the relative dualizing sheaf.
\end{lemma}

\begin{proof}
The line bundle 
\begin{equation}
  \label{eq:OmE}
  \Omega_{\X/\Y}^d\otimes
\omega_{X/Y}^{-1}\cong \OO_{X}(E)
\end{equation}
where $E=\sum n_iE_i$ is a Cartier divisor supported on
$f^{-1}D$.  In fact, we can make the choice of $E$ canonical. The
divisor $E$ is determined by its restriction to the regular locus of
$U\subseteq X$, since the complement of $U$ has codimension at least
$2$. So we may  replace $X$ by $U$. Then  $E$ is  the divisor of the canonical
map
$$\omega_{X/Y}\cong \Omega^d_{X}\otimes
f^*\omega_{Y}^{-1}\to \Omega^d_{\X/\Y}$$
Our goal is to  show that $E$ is in fact trivial as a
divisor. Since this is now a local problem, there is no loss of
generality in assuming that $Y$ is affine.
The log smoothness and saturation assumptions imply that $f$ is flat. Therefore all of the irreducible components  of
$f^{-1}D$, and in particular $E$, must map onto components of
$D$. Thus the preimage of    a general
curve $C\subset Y$ will meet all the components of $E$. The conditions of log smoothness and
saturation  and the isomorphism \eqref{eq:OmE} are stable under base
change to $C$. Therefore we may assume that $Y$ is a curve and that $D$
is  a point with local parameter $y$. The fibre $f^{-1}D$ is
reduced because $f$ is saturated. Choose a general point $p$ of an irreducible component $E_i$
of $E$. Then $E_i$ is smooth in neighbourhood of $p$ and $f^*y$ gives
a defining equation for it. Thus we may choose coordinates $x_1=f^*y,
x_2,\ldots x_n$  at $p$.  Then a local generator of
$\Omega_{\X/\Y}^d$ at $p$ is given by
$$(d\log x_1\wedge dx_2\wedge \ldots \wedge dx_n)\otimes (d\log
y)^{-1} = (d x_1\wedge dx_2\wedge \ldots \wedge dx_n)\otimes
(dy)^{-1}$$
which coincides with a generator of $\omega_{X/Y}$. Thus $E$ is trivial.

\end{proof}

\begin{thm}[Koll\'ar]
If $X$ smooth and $f:X\to Y$ is projective then
  $$\R f_* \omega_{X} \cong \bigoplus_i R^i f_*\omega_{X}[-i]$$
\end{thm}

\begin{proof}
The first step is to apply the log version of the  weak semistable
reduction theorem of Abramovich and
 Karu \cite{ak}. Although Illusie and Temkin
 \cite[thm 3.10]{it} have given such a version, it is a bit simpler
 to use the original form of the  theorem, and then translate the conclusion.
The  theorem  yields a diagram
 $$\xymatrix{
 X'\ar[d]^{f'}\ar[r]^{\pi} & X\ar[d]^{f} \\ 
 D\subset Y'\ar[r]^{p} & Y
}$$
where $p$ is generically finite, $Y'$ is smooth,  $D$ is a divisor with simple
normal crossings, $X'$ is birational to the fibre product, and $f'$ is
weakly semistable. Then, as we explained in example \ref{ex:semi3}, $f'$ is log smooth exact vertical and
saturated with respect to the log schemes $\X'$ defined by
$f^{-1}D$ and $\Y'$ defined by $D$. Furthermore,
 it is known that $X'$ has rational Gorenstein singularities \cite[lemma 6.1]{ak}.

By   lemma \ref{lemma:2}  and theorem~\ref{thm:decomp2}, we obtain
$$\R f'_* \omega_{ X'/Y'} =\bigoplus R^if'_*\omega_{ X'/Y'}[-i]$$
and therefore
\begin{equation}
  \label{eq:Rf}
  \R f'_* \omega_{X'} =\bigoplus R^if'_*\omega_{X'}[-i]. 
\end{equation}
 Fix a resolution of singularities $g:\tilde X\to X'$. Then $\R g_*\omega_{\tilde
  X}= g_*\omega_{\tilde X}= \omega_{X'}$, where the first equality follows  from the Grauert-Riemenschneider
vanishing theorem, and the second from the fact that $X'$ has rational
singularities. This together with \eqref{eq:Rf} shows that
\begin{equation}
  \label{eq:2}
  \R (f'\circ g)_* \omega_{\tilde X} =\bigoplus R^i(f'\circ g)_*\omega_{\tilde X}[-i]. 
\end{equation}
We have an inclusion $(\pi\circ g)^*\omega_X\subset
\omega_{\tilde X}$ which gives an injection
$$\sigma:\omega_{X}\hookrightarrow (\pi\circ g)*(\pi\circ g)^*\omega_X \hookrightarrow 
(\pi\circ g)_*\omega_{\tilde X}$$
The map $\sigma$ splits  the normalized Grothendieck trace
$$\tau=\frac{1}{\deg X'/X}tr:\R (\pi\circ g)_*\omega_{\tilde X} \cong (\pi\circ g)_*\omega_{\tilde
  X}\to \omega_X $$
It follows that $\omega_X$ is a direct summand of $(\pi\circ g)_*\omega_{\tilde
  X}$, and that this relation persists after applying a direct image functor.
Therefore  applying $\tau$ to \eqref{eq:2} yields
$$\R f_* \omega_{X} =\bigoplus R^if_*\omega_{X}[-i]. $$
\end{proof}

\end{document}